\documentclass{amsart}
\usepackage{amssymb}
\usepackage{amsmath}
\usepackage{amsthm}
\usepackage[all]{xy}
\usepackage{eucal}
\usepackage{color}

\SelectTips{cm}{}
\theoremstyle{plain}
\newtheorem{thm}{Theorem}[section]

\newtheorem{lem}[thm]{Lemma}
\newtheorem{prop}[thm]{Proposition}

\theoremstyle{definition}
\newtheorem{defn}[thm]{Definition}

\theoremstyle{remark}
\newtheorem{rem}[thm]{Remark}

\newcommand{\Z}{{\mathbb Z}}

\newcommand{\M}{\mathcal{M}}
\newcommand{\N}{\mathcal{N}}
\newcommand{\V}{\mathcal{V}}

\newcommand{\sSets}{{\rm sSets}}
\newcommand{\Sp}{{\rm Sp}^{\Sigma}}

\newcommand{\Coll}{\mathcal{C}{\rm{oll}}}
\newcommand{\Oper}{{\mathcal{O}\rm{per}}}
\newcommand{\Hom}{\mathop{\textrm{\rm Hom}}}
\newcommand{\Map}{\mathop{\textrm{\rm Map}}}
\newcommand{\map}{\mathop{\textrm{\rm map}}}
\newcommand{\End}{\mathop{\textrm{\rm End}}}

\numberwithin{equation}{section}

\begin{document}

\title[Transfer of algebras over operads along Quillen adjunctions]{Transfer of algebras over operads along \\ derived Quillen adjunctions}
\author[J.\ J.\ Guti\'errez]{Javier J.\ Guti\'errez}
\thanks{The author was supported by the MEC-FEDER grant MTM2010-15831, and by the
Generalitat de Catalunya as a member of the team 2009~SGR~119.}
\address{Departament d'\`Algebra i Geometria, Universitat de Barcelona.
Gran Via de les Corts Catalanes, 585, 08007 Barcelona, Spain}
\email{javier.gutierrez@ub.edu}


\begin{abstract}
Let $\V$ be a cofibrantly generated monoidal model category and
let $\M$ be a monoidal $\V$-model category. Given a cofibrant
$C$-coloured operad $\mathcal{P}$ in $\V$, we give sufficient
conditions for the fibrant replacement and cofibrant replacement
functors in $\M^C$ to preserve $\mathcal{P}$-algebra structures.
In particular, we show how $\mathcal{P}$-algebra structures can be
transferred along derived Quillen adjunctions of monoidal
$\V$-model categories, and we apply this result to the Quillen
adjunctions defined by enriched Bousfield localizations and
colocalizations on $\M$. As an application, we prove that in the
category of symmetric spectra the $n$\nobreakdash-connective cover
functor preserves $A_{\infty}$ and $E_{\infty}$ module spectra
over connective ring spectra, for every $n\in \Z$.
\end{abstract}
\maketitle

\section{Introduction}

A systematic study about the preservation of algebras over coloured operads in simplicial monoidal model categories was undertaken in \cite{CGMV10}, where sufficient conditions were given for a homotopical localization functor to preserve algebras over cofibrant operads. The examples included $A_{\infty}$ and $E_{\infty}$ structures in the category of topological spaces and in the category of symmetric spectra, but also encompassed a number of previous results about preservation of structures under localizations, such as homotopy associative $H$-spaces, loop spaces or infinite loop spaces \cite{Bou96}, \cite{Far96}, and homotopy ring spectra and homotopy module spectra~\cite{CG05}. Similar techniques have been used recently to prove that the slice functor in motivic stable homotopy preserves motivic $E_{\infty}$ ring spectra and motivic $E_{\infty}$~modules~\cite{GRSO10}.

In this paper, we study how the classes of algebras over cofibrant
operads can be transferred along derived Quillen adjunctions of
monoidal $\mathcal{V}$-categories, generalizing the results of
\cite{CGMV10} in two ways. First, we extend the framework to
monoidal $\mathcal{V}$-categories, where $\mathcal{V}$ is a closed
symmetric monoidal category. These are categories enriched,
tensored and cotensored over $\mathcal{V}$ and with a monoidal
structure compatible with the one in $\mathcal{V}$ in a suitable
way. And second, we consider more general transformations, namely
monads and comonads associated to derived Quillen adjunctions of
monoidal $\mathcal{V}$-model categories. Enriched Bousfield
localizations and enriched Bousfield colocalizations, as defined
in~\cite{Bar10}, are particular instances of this situation.

We will make use of the existence of a cofibrantly generated \emph{semi} model structure on the category of $C$-coloured operads in any cofibrantly generated monoidal model category $\mathcal{V}$, transferred from the model structure on the category of $C$\nobreakdash-coloured collections in $\mathcal{V}$ along the free-forgetful adjunction. Thus, fibrations and weak equivalences of operads are defined at the level of the underlying collections. Semi model structures appeared in \cite{Hov98} for the study of categories of algebras over commutative monoids, and were firstly used for categories of operads in \cite{Spi}. In a semi model structure all the axioms for a model structure are satisfied, except for the lifting and factorization axiom which are required to hold only for morphisms with cofibrant domain.

We prove the following statement. Let $\mathcal{P}$ be a cofibrant operad in a cofibrantly generated monoidal model category $\mathcal{V}$. Let $\mathcal{M}$ and $\mathcal{N}$ be monoidal $\mathcal{V}$-model categories  with cofibrant units, and $F\colon\M\rightleftarrows\N\colon U$ a monoidal Quillen $\V$-adjunction. Let
$$
\xymatrix{
\mathbb{L}F\colon {\rm Ho}(\M) \ar@<3pt>[r] & \ar@<2pt>[l] {\rm Ho}(\N)\colon \mathbb{R}U
}
$$
denote the derived adjunction. If $X$ is a $\mathcal{P}$-algebra with $X$ cofibrant in $\mathcal{M}$, then
$\mathbb{R}U\mathbb{L}FX$ admits a homotopy unique $\mathcal{P}$-algebra structure such that the unit of the derived adjunction $X\rightarrow
\mathbb{R}U\mathbb{L}FX$ is a map of $\mathcal{P}$-algebras.
Dually, if $X$ is a $\mathcal{P}$\nobreakdash-algebra with $X$ fibrant in $\mathcal{M}$, then $\mathbb{L}F\mathbb{R}UX$ admits a homotopy unique $\mathcal{P}$\nobreakdash-algebra structure such that the counit of the derived adjunction $\mathbb{L}F\mathbb{R}UX\rightarrow X$ is a map of $\mathcal{P}$\nobreakdash-algebras.
See Theorem~\ref{der_fib} and Theorem~\ref{der_cof} below for a
more general variant of these statements.

Our theorems specialize to the following situation. If, in the
above adjunction, $\mathcal{N}=\mathcal{M}_{\mathcal{L}}$ is the
enriched localized model structure with respect to a set of maps
$\mathcal{L}$, then the identity functors form a Quillen
adjunction and the unit of the derived adjunction is precisely the
$\mathcal{L}$-localization map. Similarly if
$\mathcal{M}=\mathcal{N}^{\mathcal{K}}$ is the enriched
colocalized model structure with respect to a set of objects
$\mathcal{K}$, then the identity functors form a Quillen
adjunction and the counit of the derived adjunction is the
$\mathcal{K}$-colocalization map. Thus, we give sufficient
conditions for localizations and colocalizations to preserve
$\mathcal{P}$-algebra structures. In the case of localization
functors, these results improve the ones in \cite[\S6]{CGMV10}.

Since we are dealing with algebras over $C$-coloured operads, sometimes we will need to apply our localization or colocalization functors on some subset of the components. For this purpose, we introduce \emph{ideals} and \emph{coideals} on the set of colours of a given $C$-coloured operad. These notions will be useful for discussing colocalization of modules over monoids, for example. See Theorem~\ref{mainthm_ideal} and Theorem~\ref{mainthm_coideal} for precise statements.

We illustrate our results with the following example. For a given $n\in\Z$, let $K_{n}$ denote the $(n-1)$-connective cover functor in the category of symmetric spectra, i.e., the enriched Bousfield colocalization functor on the category of symmetric spectra enriched over simplicial sets with respect to $\Sigma^{n}S^0$, where $S^0$ denotes the sphere spectrum. We prove that if $R$ is an $E_{\infty}$ ring and $M$ is an $E_{\infty}$ module with $R$ and $M$ fibrant as spectra, then the $(n-1)$-connective cover $K_{n} M$ has a homotopy unique $E_{\infty}$ module structure over the $E_{\infty}$ ring $K_0R$ such that the colocalization map is a map of $E_{\infty}$ modules. The
same statement is true for $A_{\infty}$ modules over $A_{\infty}$ rings.

\medskip
\noindent\textbf{Acknowledgements}. I would like to thank Denis-Charles Cisinski for sharing with me some of his ideas on the topic of this paper and for many helpful discussions.

\section{Background on model categories}
In this section, we recall the definition of a derived Quillen
adjunction between two model categories and the basic notions of
enriched category theory in the context of model categories; see
\cite[\S4]{Hov99} for details.

A \emph{model category} is a category $\mathcal{M}$ with a model
structure as in~\cite{Qui67}, \cite{Hov99} or~\cite{Hir03}. We
will assume the existence of functorial factorizations as part of the
definition of a model category. Thus, every model category $\M$
comes equipped with a \emph{cofibrant replacement functor}
$Q\colon\M\rightarrow\M$ together with a natural
transformation $i\colon Q\rightarrow  {\rm Id}$ such that
$i_X\colon QX\rightarrow X$ is a trivial fibration and $QX$ is
cofibrant for every $X$ in $\M$. Dually, there exist a
\emph{fibrant replacement functor} $R\colon\M\rightarrow\M$
together with a natural transformation $j\colon{\rm
Id}\rightarrow R$ such that $j_X\colon X\rightarrow RX$ is
a trivial cofibration and $RX$ is fibrant for every $X$ in $\M$.

The \emph{homotopy category} of a model category $\M$ will be denoted by ${\rm Ho}(\M)$ and is the category obtained by inverting the weak equivalences in $\M$. It is equivalent to the quotient category of the fibrant and cofibrant objects modulo the homotopy relation.

\subsection{Quillen functors}
An adjoint pair $F\colon
\mathcal{M}\leftrightarrows\mathcal{N}\colon U$ between model
categories is called a \emph{Quillen adjunction} if $F$ preserves
cofibrations and trivial cofibrations or, equivalently, if $U$
preserves fibrations and trivial fibrations. The left adjoint of a
Quillen adjunction is called a \emph{left Quillen functor} and the
right adjoint is called a \emph{right Quillen functor}.

By Ken Brown's lemma~\cite[Lemma 1.1.12]{Hov99}, every left Quillen functor sends weak equivalences between cofibrant objects to weak equivalences and, dually, every right Quillen functor sends weak equivalences between fibrant object to weak equivalences.

If $F\colon \M\rightarrow \N$ is a left Quillen functor, then the \emph{left derived functor}
$$
\mathbb{L}F\colon {\rm Ho}(\M)\longrightarrow {\rm Ho}(\N)
$$
is defined by $\mathbb{L}F=F\circ Q$, where $Q$ denotes the
cofibrant replacement functor in~$\M$. Similarly, if $U\colon
\N\rightarrow \M$ is a right Quillen functor, then the \emph{right
derived functor}
$$
\mathbb{R}U\colon {\rm Ho}(\N)\longrightarrow {\rm Ho}(\M)
$$
is defined as $\mathbb{R}U=U\circ R$, where $R$ denotes the fibrant replacement functor.

Every Quillen adjunction $F\colon \M\leftrightarrows\N\colon U$ induces a \emph{derived adjunction} between the homotopy categories
$$
\xymatrix{
\mathbb{L}F\colon {\rm Ho}(\M)\ar@<3pt>[r] & \ar@<2pt>[l] {\rm Ho}(\N)\colon \mathbb{R}U
}
$$

\subsection{Enriched monoidal model categories}
\label{enrichedmonmod}
Given two model categories $\V$ and $\M$, a \emph{Quillen adjunction of two variables} is a triple of bifunctors $-\odot -\colon \V\times\M\rightarrow \M$, $\Hom(-,-)\colon \M^{\rm op}\times \M\rightarrow \V$ and $(-)^{(-)}\colon \V^{\rm op}\times M\longrightarrow \M$ together with natural isomorphisms
$$
\V(A, \Hom(X, Y))\cong \M(A\odot X, Y)\cong \M(X, Y^A)
$$
for every $X$ and $Y$ in $\M$ and every $A$ in $\V$, such that if $f\colon A\rightarrow B$ is a cofibration in $\V$ and $g\colon X\rightarrow Y$ is a cofibration in $\M$, then the \emph{pushout-product} of $f$ and $g$
\begin{equation}
(A\odot Y)\coprod_{A\odot X} (B\odot X)\longrightarrow B\odot Y
\label{pushproduct}
\end{equation}
is a cofibration that is trivial if either $f$ or $g$ is trivial. By adjointness, condition~(\ref{pushproduct}) is equivalent to
the following analog of Quillen's (SM7) axiom for simplicial model categories~\cite[II.\S2]{Qui67}: If $f\colon X\rightarrow Y$ is a cofibration in $\M$ and $g\colon W\rightarrow Z$ is a fibration in $\M$, then the morphism
\begin{equation}
\Hom(Y,W)\longrightarrow\Hom(Y,Z)\times_{\Hom(X,Z)}\Hom(X, W)
\label{enrichedSM7}
\end{equation}
is a fibration in $\V$ that is trivial if either $f$ or $g$ is trivial.

Observe that if $A$ is cofibrant in $\V$, then the functor $A\odot
-$ is a left Quillen functor, with right adjoint $(-)^A$. Also, if
$X$ is cofibrant in $\M$, then $-\odot X$ is a left Quillen
functor with right adjoint $\Hom(X,-)$.

A \emph{monoidal model category} is a closed symmetric monoidal category $\V$ with tensor product $\otimes$, unit object $I$ and internal hom denoted by $\Hom_{\V}$ together with a model structure that satisfies the following two axioms:
\begin{itemize}
\item[(i)] \emph{Pushout-product axiom}. The functors $-\otimes-$, $\Hom_{\V}(-,-)$ and $\Hom_{\V}(-,-)$ form a Quillen adjunction of two variables.
\item[(ii)] \emph{Unit axiom}. The natural map
$$
q\otimes 1\colon \widetilde{I}\otimes A \longrightarrow I\otimes A
$$
is a weak equivalence in $\V$ for every cofibrant $A$, where $q\colon \widetilde{I}\rightarrow I$
denotes any cofibrant approximation of the unit of~$\V$. (A \emph{cofibrant approximation} to $X$ is a cofibrant object $\widetilde X$ and a weak equivalence $\widetilde{X}\rightarrow X$.)
\end{itemize}

A direct consequence of the pushout-product axiom is that the
tensor product of two cofibrations with cofibrant domains is again
a cofibration. Note that the unit axiom holds trivially if the
unit of $\V$ is cofibrant.

A functor $F\colon \V\rightarrow \V'$ between symmetric monoidal categories is called \emph{symmetric monoidal} if it is equipped with a unit $\alpha\colon I_{\V'}\rightarrow F(I_{\V})$ and a binatural transformation $m\colon F(-)\otimes_{\V'} F(-)\rightarrow F(-\otimes_{\V} -)$ satisfying the usual associativity, symmetry and unit conditions. A symmetric monoidal functor is called  \emph{strong} if the structure maps are isomorphisms.

Given two (strong) symmetric monoidal functors $F$ and
$F'\colon\V\rightarrow \V'$, a \emph{symmetric monoidal natural
transformation} is a natural transformation $\tau\colon
F\rightarrow F'$ satisfying $\alpha_{F'}=\tau_{I_{\V}}\circ
\alpha_{F}$, and such that the diagram
$$
\xymatrix{
FX\otimes_{\V'} FY\ar[r]^{m_{F}}\ar[d]_{\tau_X\otimes\tau_Y} & F(X\otimes_{\V} Y)\ar[d]^{\tau_{X\otimes Y}} \\
F'X\otimes_{\V'} F'Y\ar[r]_{m_{F'}} & F'(X\otimes_{\V} Y) }
$$
commutes for every $X$ and $Y$ in $\V$.

Let $\V$ and $\V'$ be two monoidal model categories. Recall from
\cite[Definition~4.2.16]{Hov99} that a Quillen adjunction $F\colon
\V\rightleftarrows \V'\colon U$ is called a \emph{monoidal Quillen
adjunction} if the left adjoint $F$ is strong symmetric monoidal,
and for any cofibrant approximation of the unit
$\widetilde{I}_{\V}\rightarrow I_{\V}$, the map $
F(\widetilde{I}_{\V})\rightarrow F(I_{\V})\cong I_{\V'}$ is a weak
equivalence. Note that in this case the right adjoint $U$ is
automatically symmetric monoidal, since $F$ is strong symmetric
monoidal.

\begin{defn}
Let $\V$ be a monoidal model category. A \emph{$\V$-model
category} is a category $\M$ enriched, tensored and cotensored
over $\V$ together with a model structure such that the tensor,
enrichment and cotensor define a Quillen adjunction of two
variables, and for every cofibrant object $X$ in $\M$, the map
$$
\widetilde{I}\otimes X\longrightarrow I\otimes X
$$
is a weak equivalence in $\M$, where $\widetilde{I}\rightarrow I$ denotes any cofibrant approximation of the unit.
\end{defn}

\begin{defn}
Let $\V$ be a monoidal model category. A \emph{monoidal $\V$-model
category} is a monoidal model category $\M$ together with a
monoidal Quillen adjunction $i\colon \V\rightleftarrows \M\colon
r$.
\end{defn}

Every monoidal model category $\V$ is itself a $\V$-model category
with the tensor product and the internal hom of $\V$, and it is
also a monoidal $\V$-model category trivially by taking $F=U={\rm
Id}$.

The following result is a direct consequence of the pushout-product axiom in $\V$ and in $\M$, and the fact that $i$ is a left Quillen functor and $r$ a right Quillen functor.
\begin{prop}
Let $\M$ be a monoidal $\V$-model category with associated monoidal Quillen adjunction $i\colon \V\rightleftarrows \M\colon r$. Then, the functors $i(-)\otimes -$, $r(\Hom_{\M}(-,-))$ and $\Hom_{\M}(i(-),-)$ form a Quillen adjunction of two variables. $\hfill\qed$
\end{prop}

A \emph{monoidal Quillen $\V$-adjunction} between two monoidal
$\V$-model categories $\M$ and $\N$ is a monoidal Quillen
adjunction $F\colon\M\rightleftarrows \N\colon U$ that respects
the action of $\V$, hence $F$ is a strong symmetric monoidal
functor together with a monoidal natural isomorphism $\rho\colon
F\circ i_{\M}\rightarrow i_{\N}$.

Throughout the rest of the paper, given a monoidal $\V$-model
category $\M$ and objects $X$ and $Y$ in $\M$, and $A$ in $\V$, we
will denote by $A\otimes X$ the tensor $F(A)\otimes X$, and by
$\Hom(X,Y)$ the enrichment $U(\Hom_{\M}(X,Y))$.

\section{Coloured operads and algebras in enriched model categories}

Coloured operads can be defined in any closed symmetric monoidal category $\V$ and its algebras have sense in any monoidal $\V$-category $\M$. In the first part of this section, we review the terminology and basic properties of coloured operads and its algebras (see \cite{BV73}, \cite[\S2]{EM06}, \cite{Lei04}). In the second part, we describe the (semi) model structure of the category of $C$-coloured operads in a cofibrantly generated monoidal model category.

Throughout this section, $\mathcal{V}$ will denote a cocomplete closed symmetric monoidal category with tensor product $\otimes$ and unit $I$. We will denote by $0$ an initial object of $\mathcal{V}$.

\subsection{Coloured operads and algebras}
 Let $C$ be any set, whose elements will be called \emph{colours}.
A \emph{$C$\nobreakdash-coloured collection} $\mathcal{K}$ in $\mathcal{V}$ consists of objects $\mathcal{K}(c_1,\ldots, c_n;c)$ in $\mathcal{V}$, one  for each $(n+1)$-tuple of
colours $(c_1,\ldots, c_n,c)$ and $n\ge 0$, equipped with a right action of the
symmetric groups, i.e., for each $\alpha\in \Sigma_n$ there are maps
$$
\alpha^*\colon \mathcal{K}(c_1,\ldots, c_n;c)\longrightarrow
\mathcal{K}(c_{\alpha(1)},\ldots, c_{\alpha(n)};c),
$$
where $\Sigma_n$ is the symmetric group on $n$ letters (the group $\Sigma_n$ denotes the trivial group if $n=0$ or $n=1$).

A \emph{morphism} of $C$\nobreakdash-coloured collections $\varphi\colon
\mathcal{K}\longrightarrow \mathcal{L}$ is a family of maps
$$
\varphi_{c_1,\ldots,c_n;c}\colon \mathcal{K}(c_1,\ldots,c_n;c)\longrightarrow
\mathcal{L}(c_1,\ldots, c_n; c)
$$
in $\mathcal{V}$, ranging over all $n\ge 0$ and all $(n+1)$-tuples
$(c_1,\ldots, c_n,c)$, and compatible with the action of the
symmetric groups. The category of $C$\nobreakdash-coloured collections in $\mathcal{V}$
will be denoted by $\Coll_C(\mathcal{V})$.

A \emph{$C$\nobreakdash-coloured operad} $\mathcal{P}$ in $\mathcal{V}$ is a $C$\nobreakdash-coloured
collection equip\-ped with unit maps $I\rightarrow \mathcal{P}(c;c)$ for
every $c\in C$ and, for every $(n+1)$-tuple of colours
$(c_1,\ldots, c_n,c)$ and $n$ given tuples
\[
(a_{1,1},\ldots, a_{1,k_1}, c_1),\ldots, (a_{n,1},\ldots,
a_{n,k_n}, c_n),
\]
a \emph{composition product} map
$$\xymatrix{
\mathcal{P}(c_1,\ldots, c_n;c)\otimes \mathcal{P}(a_{1,1},\ldots,
a_{1,k_1};c_1)\otimes\cdots\otimes \mathcal{P}(a_{n,1},\ldots,
a_{n,k_n};c_n)\ar[d]
\\ \mathcal{P}(a_{1,1},\ldots,a_{1,k_1},a_{2,1},\ldots,a_{2,k_2},\ldots,a_{n,1},\ldots,a_{n,k_n};c),
}
$$
that is compatible with the action of the symmetric groups and subject
to the usual associativity and unitary compatibility relations; see, for
example, \cite[\S 2]{EM06}.

A \emph{morphism of $C$\nobreakdash-coloured operads} is a morphism of the
underlying $C$\nobreakdash-co\-lou\-red collections that is compatible
with the unit maps and the composition product maps.
The category of $C$\nobreakdash-coloured operads in $\mathcal{V}$ will be denoted by ${\Oper}_C(\mathcal{V})$. There is a free-forgetful adjunction
$$
\xymatrix{
F \colon \Coll_C(\mathcal{V}) \ar@<3pt>[r] & \ar@<2pt>[l] \Oper_C(\mathcal{V})\colon U
}
$$
where $U$ is the forgetful functor, and the left adjoint is the
free $C$-coloured operad generated by a collection.

If $C=\{c\}$, then a $C$-coloured operad $\mathcal{P}$ is just an ordinary operad, where one writes $\mathcal{P}(n)$
instead of $\mathcal{P}(c,\ldots,c;c)$ with $n$ inputs, for every $n\ge 0$.

Let $\M$ be a  monoidal $\V$-category. Algebras in $\M$ over coloured operads in $\V$ are defined in the following way. We denote by $\M^C$ the product category of copies of $\M$ indexed by the set of colours $C$, that is, $\M^C=\prod_{c\in C}\M$. For every object $\mathbf{X}=(X(c))_{c\in C}$ in $\M^C$, the \emph{endomorphism coloured operad} $\End({\mathbf{X}})$ of~$\mathbf{X}$ is the $C$-coloured operad in $\V$ defined by
$$
\End(\mathbf{X})(c_1,\ldots, c_n; c)=\Hom(X(c_1)\otimes\cdots\otimes X(c_n),X(c)),
$$
where $X(c_1)\otimes\cdots\otimes X(c_n)$ is meant to be $I$ if $n=0$. The composition product is ordinary
composition and the $\Sigma_n$-action is defined by permutation of the factors.

Let $\mathcal{P}$ be any $C$-coloured operad in $\mathcal{V}$. A \emph{$\mathcal{P}$-algebra} in $\M$ is an object $\mathbf{X}=(X(c))_{c\in C}$ of
$\mathcal{\M}^C$ together with a morphism
$
\mathcal{P}\rightarrow \End(\mathbf{X})
$
of $C$\nobreakdash-coloured operads in $\V$.

A \emph{map of $\mathcal{P}$-algebras} $\mathbf{f}\colon\mathbf{X}\rightarrow \mathbf{Y}$ is a family of maps
$(f_c\colon X(c)\rightarrow Y(c))_{c\in C}$ such that the following diagram of $C$-coloured collections in $\V$
$$
\xymatrix{
\mathcal{P}\ar[r] \ar[d] & \End(\mathbf{X}) \ar[d] \\
\End(\mathbf{Y}) \ar[r]& \Hom(\mathbf{X},\mathbf{Y})
}
$$
commutes, where the top and left arrows are the given $\mathcal{P}$-algebra structures on $\mathbf{X}$ and $\mathbf{Y}$ respectively, the $C$-coloured collection $\Hom(\mathbf{X},\mathbf{Y})$ is defined as
$$
\Hom(\mathbf{X},\mathbf{Y})(c_1,\ldots, c_n;c)=\Hom(X(c_1)\otimes\cdots\otimes X(c_n), Y(c)),
$$
and the right and bottom maps are induced by the maps $f_c$.

If the category $\mathcal{V}$ has pullbacks, then a map $\mathbf{f}$ of $\mathcal{P}$-algebras can be seen as a map of $C$-coloured operads
$$
\mathcal{P}\longrightarrow \End(\mathbf{f}),
$$
where the $C$-coloured operad $\End(\mathbf{f})$ is obtained as the pullback of the diagram of $C$-coloured collections
in $\V$
\begin{equation}
\xymatrix{
\End(\mathbf{f})\ar@{.>}[r] \ar@{.>}[d] & \End(\mathbf{X}) \ar[d] \\
\End(\mathbf{Y}) \ar[r]& \Hom(\mathbf{X},\mathbf{Y}).
}
\label{end_f}
\end{equation}
Note that $\End(\mathbf{f})$ inherits indeed a $C$-coloured operad structure from the $C$-coloured operads $\End(\mathbf{X})$ and $\End(\mathbf{Y})$.

Given a $C$-coloured operad $\mathcal{P}$ and an object $\mathbf{X}=(X(c))_{c\in C}$ in $\M^C$, the
\emph{restricted endomorphism operad} $\End_{\mathcal{P}}(\mathbf{X})$ is defined as follows:
\begin{equation}
\End\nolimits_{\mathcal{P}}(\mathbf{X})(c_1,\ldots, c_n; c)=
\left\{
\begin{array}{ll}
\End(\mathbf{X})(c_1,\ldots, c_n;c) & \mbox{if }\mathcal{P}(c_1,\ldots, c_n;c)\ne 0. \\[0.2cm]
0 & \mbox{otherwise}.
\end{array}
\right.
\label{endp}
\end{equation}
There is a canonical inclusion of $C$-coloured operads
$\End_{\mathcal{P}} (\mathbf{X})\rightarrow \End(\mathbf{X})$,
and thus every morphism $\mathcal{P}\rightarrow
\End(\mathbf{X})$ of $C$-coloured operad factors uniquely through
the restricted endomorphism operad
$\End_{\mathcal{P}}(\mathbf{X})$. Hence, a $\mathcal{P}$-algebra
structure on $\mathbf{X}$ is, in fact, given by a morphism of
$C$-coloured operads $\mathcal{P}\rightarrow
\End_{\mathcal{P}}(\mathbf{X})$

Similarly, if $\mathbf{X}$ and $\mathbf{Y}$ are objects of $\M^C$ and $\mathcal{P}$ is a $C$-coloured operad in $\V$, we denote by
$\Hom_{\mathcal{P}} (\mathbf{X},\mathbf{Y})$ the $C$-coloured collection defined as
\begin{equation}
\Hom\nolimits_{\mathcal{P}}(\mathbf{X},\mathbf{Y})(c_1,\ldots, c_n; c)=
\left\{
\begin{array}{ll}
\Hom(\mathbf{X},\mathbf{Y})(c_1,\ldots, c_n;c) &\mbox{if }\mathcal{P}(c_1,\ldots, c_n;c)\ne 0,\\[0.2cm]
0 & \mbox{otherwise}, \\
\end{array}
\right.
\label{homp}
\end{equation}
and, for a morphism $\mathbf{f}: \mathbf{X}\rightarrow\mathbf{Y}$ of $\mathcal{P}$-algebras, we denote by $\End_{\mathcal{P}}(\mathbf{f})$ the pullback (when it exists) of the restricted
endomorphism operads of $\mathbf{X}$ and $\mathbf{Y}$ over $\Hom_{\mathcal{P}}(\mathbf{X},\mathbf{Y})$, as in~(\ref{end_f}).

\subsection{The semi model structure for coloured operads}
Let $\V$ be a cofibrantly generated monoidal model category. Then,
the category of $C$-coloured collections $\Coll_C(\V)$ admits a
cofibrantly generated model structure, in which a morphism
$\mathcal{K}\rightarrow\mathcal{L}$ is a fibration or a weak
equivalence if $\mathcal{K}(c_1,\ldots, c_n; c)\rightarrow
\mathcal{L}(c_1,\ldots, c_n; c)$ is a fibration or a weak
equivalence in $\V$, respectively, for any $(n+1)$-tuple of
colours $(c_1,\ldots, c_n,c)$. Under suitable conditions (see
\cite{Cra95}, \cite[Theorem 11.3.2]{Hir03}) this model structure
can be transferred along the free-forgetful adjunction
\begin{equation}
\xymatrix{
F \colon\Coll_C(\mathcal{V}) \ar@<3pt>[r] & \ar@<2pt>[l] \Oper_C(\mathcal{V}) \colon U
}
\label{freeforgetful}
\end{equation}
to define a cofibrantly generated model structure on the category of $C$-coloured operads in which a morphism of $C$-coloured operads is a fibration or a weak equivalence if its underlying morphism of $C$-coloured collections is a fibration or a weak equivalence, respectively. The class of cofibrations is defined as the class of morphisms with the left lifting property with respect to trivial fibrations. These conditions hold if we assume that the unit of $\V$ is cofibrant, the existence of a symmetric monoidal fibrant replacement functor in $\V$ and the existence of an interval with a coassociative and cocommutative comultiplication~\cite[Theorem~2.1]{BM07} (see also \cite[Theorem~1.1]{JY} for a generalization to coloured PROPs).

For example, the categories of simplicial sets, $k$-spaces or
chain complexes over a commutative ring with unit satisfy these
assumptions. However, no symmetric monoidal model category of
spectra can have at the same time a cofibrant unit and a symmetric
monoidal fibrant replacement functor~\cite{Lew91}. This problem
can sometimes be avoided by considering the coloured operad whose
algebras are $C$\nobreakdash-coloured operads for a fixed sets of
colours~$C$; see \cite{GV10} for a discussion in the category of
symmetric spectra with the positive model structure.

In general, if $\V$ is only a cofibrantly generated monoidal model category without any further assumptions, then the free-forgetful adjunction~(\ref{freeforgetful}) does not provide a full model structure on $\Oper_C({\V})$, but a weaker notion called a \emph{semi model structure}. The idea of a semi model category is to assume that all axioms of model
categories hold, except for the lifting axiom and the factorization axiom,
that hold only for morphisms with cofibrant domain. Semi model categories are a particular example of \emph{cat\'egorie d\'erivable relevante} in the sense of~\cite{Cis10}; see also \cite{Rad09}.

\begin{defn}
A \emph{semi model category} structure on a category $\V$ is given
by three classes of distinguished morphisms called \emph{weak
equivalences}, \emph{fibrations} and \emph{cofibrations},
satisfying the following axioms:
\begin{itemize}
\item[{\rm (i)}] $\V$ is complete and cocomplete \item[{\rm (ii)}]
Let $f$ and $g$ be two composable morphisms. If any two among $f$,
$g$ and $g\circ f$ are weak equivalences, then so is the third
\item[{\rm (iii)}] The classes of weak equivalences, fibrations
and cofibrations are closed under retracts. \item[{\rm (iv)}] The
class of fibrations has the right lifting property with respect to
the class of trivial cofibrations with cofibrant domain. The class
of trivial fibrations has the right lifting property with respect
to the class of cofibrations with cofibrant domain. \item[{\rm
(v)}] There exist functorial factorizations of any morphism with
cofibrant domain into a cofibration followed by a trivial
fibration and also into a trivial cofibration followed by a
fibration.
\item[{\rm (vi)}] The initial object is cofibrant.
\item[{\rm (vii)}] The  classes of fibrations and trivial
fibrations are stable under pullbacks and (possibly transfinite)
composites.
\end{itemize}
\end{defn}

Usually, semi model structures are transferred by an adjunction
from a cofibrantly generated model category, as
in~(\ref{freeforgetful}). The fibrations and weak equivalences of
the semi model structure are those morphisms that are fibration
and weak equivalences in the model category when we apply the
right adjoint. Cofibrations are defined as the morphisms with they
the left lifting property with respect to the trivial fibrations.
In the semi model structures obtained this way, axioms (vi) and
(vii) are always satisfied and may be deduced from the others.

Semi model categories were first defined by Hovey in~\cite{Hov98} and applied to the case of algebras over a commutative monoid. The use of semi model structures for categories of operads is due to Spitzweck~\cite{Spi}; see also~\cite[\S12]{Fre09}. A proof of the following result for operads with one colour can be found in~\cite[Theorem~12.2A]{Fre09} (cf.\ \cite[Theorem 3.2]{Spi}). An extension to the coloured case can be proved using the same arguments. For a version in the category of PROPs, see~\cite[\S4]{Fre10}.

\begin{thm}
If $\V$ is a cofibrantly generated monoidal model category, then the model structure on $\Coll_C(\V)$ transfers along the free-forgetful adjunction to a cofibrantly generated semi model structure on $\Oper_C(\V)$, in which a morphism $\mathcal{P}\rightarrow\mathcal{Q}$ is a fibration or a weak equivalence if $\mathcal{P}(c_1,\ldots, c_n;c)\rightarrow \mathcal{Q}(c_1,\ldots, c_n; c)$
is a fibration or a weak equivalence in $\V$, respectively, for every $(c_1,\ldots, c_n, c)$. $\hfill\qed$
\end{thm}

\begin{rem}
In every semi model category, trivial fibrations have the right lifting property with respect to cofibrant objects, since the initial object of a semi model category is assumed to be cofibrant.
\end{rem}

Let $\mathcal{P}$ be a $C$-coloured operad in $\V$. We denote by $\mathcal{P}_{\infty}$ a cofibrant replacement of $\mathcal{P}$ in the semi model category $\Oper_C(\V)$, that is, $\mathcal{P}_{\infty}$ is cofibrant and there is a trivial fibration $\mathcal{P}_{\infty}\rightarrow \mathcal{P}$. The category of $\mathcal{P}_{\infty}$-algebras is called the \emph{category of homotopy $\mathcal{P}$-algebras}.

\section{Transfer of algebra structures}

Let $\V$ be a cofibrantly generated monoidal model category and
let $\Oper_C(\V)$ be the category of $C$-coloured operads in $\V$
with the transferred semi model structure. Let $\M$ be a monoidal $\V$-model category.

In this section, we study how to transfer a $\mathcal{P}$-algebra
structure on an object of $\M^C$ along a morphism, where $\M^C$ is
viewed as a model category with the product model structure.  In
particular, we give sufficient conditions for the extension of the
fibrant replacement functor and the cofibrant replacement functor
over $\M^C$ to preserve algebras over $C$-coloured operads. Our
main result is that the unit and counit of a derived Quillen
adjunction between monoidal $\V$-model categories preserve
algebras over cofibrant coloured operads. Given and object
$\mathbf{X}$ of $\M^C$ and two $\mathcal{P}$-algebra structures
$\gamma$ and $\gamma'\colon \mathcal{P}\rightarrow
\End(\mathbf{X})$ on $\mathbf{X}$, we say that they \emph{coincide
up to homotopy} if $\gamma$ and $\gamma'$ are equal in the
homotopy category of $C$-coloured operads in~$\V$.

\begin{lem}
Let $\M$ be a monoidal $\V$-model category. Let $\mathcal{P}$ be a
cofibrant $C$\nobreakdash-co\-lou\-red operad in $\V$ and
$\mathbf{f}\colon\mathbf{X}\rightarrow \mathbf{Y}$ a morphism in
$\M^C$.
\begin{itemize}
\item[{\rm (i)}] If the natural map $\End_{\mathcal{P}}(\mathbf{Y})\rightarrow \Hom_{\mathcal{P}}(\mathbf{X},\mathbf{Y})$ induced by $\mathbf{f}$ is a trivial fibration of $C$-coloured collections in $\V$, then any $\mathcal{P}$-algebra structure on $\mathbf{X}$ extends along $\mathbf{f}$ to a homotopy unique $\mathcal{P}$-algebra structure on $\mathbf{Y}$\ such that $\mathbf{f}$ is a map of $\mathcal{P}$-algebras.
\item[{\rm (ii)}] If the natural map $\End_{\mathcal{P}}(\mathbf{X})\rightarrow \Hom_{\mathcal{P}}(\mathbf{X},\mathbf{Y})$ induced by $\mathbf{f}$ is a trivial fibration of $C$-coloured collections in $\V$, then any $\mathcal{P}$-algebra structure on $\mathbf{Y}$ lifts along $\mathbf{f}$ to a homotopy unique $\mathcal{P}$-algebra structure on $\mathbf{X}$\ such that $\mathbf{f}$ is a map of $\mathcal{P}$\nobreakdash-algebras.
\end{itemize}
\label{keylem}
\end{lem}

\begin{proof}
Consider the $C$-coloured operad $\End_{\mathcal{P}}(\mathbf{f})$
defined as the following pullback in the category of $C$-coloured
collections
\begin{equation}
\xymatrix
{
\End_{\mathcal{P}}(\mathbf{f})\ar@{.>}[r]^{\rho}\ar@{.>}[d]_{\tau} & \End_{\mathcal{P}}(
\mathbf{Y}) \ar[d] \\
\End_{\mathcal{P}}(\mathbf{X})\ar[r] & \Hom_{\mathcal{P}}(\mathbf{X}, \mathbf{Y}).
}\label{endf}
\end{equation}
Under the assumptions of (i), the morphism $\tau$ is a trivial fibration since it is the pullback of a trivial fibration. Since the coloured operad $\mathcal{P}$ is cofibrant, there is a lifting
$$
\xymatrix{
 & \End_{\mathcal{P}}(\mathbf{f}) \ar[d]^{\tau} \\
\mathcal{P}\ar[r]^-{\alpha}\ar@{.>}[ur] & \End_{\mathcal{P}}(\mathbf{X}),
}
$$
where $\alpha\colon
\mathcal{P}\rightarrow\End_{\mathcal{P}}(\mathbf{X})$ is the given
$\mathcal{P}$-algebra structure on $\mathbf{X}$. Now, composing
this lifting with the map $\rho$  endows $\mathbf{Y}$ with a
$\mathcal{P}$-algebra structure such that $\mathbf{f}$ is a map of
$\mathcal{P}$-algebras. Part (ii) follows similarly by using that,
in this case, the morphism $\rho$ in~(\ref{endf}) is a trivial
fibration.

To prove homotopy uniqueness in part (i), suppose that we have two $\mathcal{P}$-algebra structures on $\mathbf{Y}$, denoted by $\gamma$ and $\gamma'\colon \mathcal{P}\rightarrow \End_{\mathcal{P}}(\mathbf{Y})$, and assume further that $\mathbf{f}$ is a map of $\mathcal{P}$-algebras for each of
them, i.e., $\gamma$ and $\gamma'$ factor through $\End_{\mathcal{P}}({\mathbf{f}})$. Now, let $\delta,\delta'\colon \mathcal{P}\rightarrow \End_{\mathcal{P}}(\mathbf{f})$ be such that $\gamma=\rho\circ\delta$ and $\gamma'=\rho\circ\delta'$. Since $\tau$ is a trivial fibration and $\tau\circ\delta=\tau\circ\delta'$, it follows that $\delta$ and $\delta'$ are equal in the homotopy category of $C$-coloured operads and hence so are $\gamma$ and $\gamma'$. Uniqueness for part (ii) is proved in the same way.
\end{proof}

\begin{defn}
Let $F$ be an endofunctor on a category $\M$. The \emph{extension
of $F$ over $\M^C$} is the endofunctor on $\M^C$, which we keep
denoting by $F$, given by $F\mathbf{X} = (FX(c))_{c\in C}$. If $F$
is equipped with an augmentation $\eta$ or a coaugmentation
$\varepsilon$, then $\eta_{\mathbf{X}}=(\eta_{X(c)})_{c\in C}$ and
$\varepsilon_{\mathbf{X}}=(\varepsilon_{X(c)})_{c\in C}$,
respectively.
\end{defn}

Our main application for Lemma~\ref{keylem} is when $\mathbf{f}$ is the natural map from $\mathbf{X}$ to the extension of a fibrant replacement $R\mathbf{X}$ or the natural map from the extension of a cofibrant replacement $Q\mathbf{X}$ to $\mathbf{X}$.
Recall that $R$ denotes the functorial fibrant replacement in $\M$, that is, a coaugmented functor $(R,j)$ such that for every $X$ in $\M$ the natural map $j_X\colon X\rightarrow RX$  is a trivial cofibration and $RX$ is fibrant.
 Dually, $Q$ denotes the functorial cofibrant replacement in $\M$, that is, an augmented functor $(Q,i)$ such that for every $X$ in $\M$ the natural map $i_X\colon QX\rightarrow X$ is a trivial fibration and $QX$ is cofibrant.

\begin{prop}
Let $\M$ be a monoidal $\V$-model category. Let $\mathcal{P}$ be a
cofibrant $C$\nobreakdash-coloured operad in $\V$ and consider the
extension of the fibrant replacement functor $R$ over $\M^C$. Let
$\mathbf{X}$ be a $\mathcal{P}$-algebra such that $X(c)$ is
cofibrant in $\M$ for every $c\in C$. Then $R\mathbf{X}$ admits a
$\mathcal{P}$-algebra structure such that the natural map
$j_{\mathbf{X}}\colon \mathbf{X}\rightarrow R\mathbf{X}$ is a map
of $\mathcal{P}$-algebras. \label{fib_rep_alg}
\end{prop}

\begin{proof}
For all $n\ge 0$ and all $(c_1,\ldots, c_n;c)$ such that $\mathcal{P}(c_1,\ldots,c_n; c)\ne 0$ the induced map
$$
\xymatrix{
\Hom(RX(c_1)\otimes\cdots\otimes RX(c_n), RX(c)) \ar[d]^{(j_{\mathbf{X}})^*} \\
\Hom(X(c_1)\otimes\cdots\otimes X(c_n), R X(c))
}
$$
is a trivial fibration in $\V$. Indeed, the map
$$
X(c_1)\otimes\cdots\otimes X(c_n)\longrightarrow RX(c_1)\otimes\cdots\otimes RX(c_n)
$$
is a trivial cofibration by the fact that $X(c)$ is cofibrant for
all $c\in C$ and the pushout-product axiom in $\M$. Now,
$(j_{\mathbf{X}})^*$  is a trivial fibration
by~(\ref{enrichedSM7}) since $RX(c)$ is fibrant. Hence the
morphism of $C$-coloured collections
$$
\End\nolimits_{\mathcal{P}}(R\mathbf{X})\longrightarrow \Hom\nolimits_{\mathcal{P}}(\mathbf{X}, R\mathbf{X})
$$
is a trivial fibration. The result follows now from
Lemma~\ref{keylem}(i).
\end{proof}

\begin{rem}
The cofibrancy condition on the objects $X(c)$ is needed to prove that the map
$$
X(c_1)\otimes\cdots\otimes X(c_n)\longrightarrow RX(c_1)\otimes\cdots\otimes RX(c_n)
$$
is a trivial cofibration. In general, the tensor product of two trivial cofibrations is not a trivial cofibration. The pushout-product axiom implies that this is the case if the trivial cofibrations involved have cofibrant domains or, more generally, when the domains of the generating trivial cofibrations of $\M$ are cofibrant.
\end{rem}

\begin{prop}
Let $\M$ be a monoidal $\V$-model category. Let $\mathcal{P}$ be a
cofibrant $C$\nobreakdash-coloured operad in $\V$ and consider the
extension of a cofibrant replacement $Q$ over $\M^C$. Let
$\mathbf{X}$ be a $\mathcal{P}$-algebra in $\M$. If the induced
map
\begin{equation}
\Hom(I, QX(c))\longrightarrow \Hom(I, X(c)),
\label{unit_assumption}
\end{equation}
where $I$ denotes the unit of $\M$, is a trivial fibration for
every $c\in C$ such that $\mathcal{P}(\,;c)\ne 0$, then
$Q\mathbf{X}$ admits a $\mathcal{P}$\nobreakdash-algebra structure
such that the natural map $i_{\mathbf{X}}\colon
Q\mathbf{X}\rightarrow \mathbf{X}$ is a map of
$\mathcal{P}$-algebras. \label{cof_rep_alg}
\end{prop}

\begin{proof}
The result is proved in the same way as Proposition \ref{fib_rep_alg}, but using Lem\-ma~\ref{keylem}(ii). In this case we have to check that the induced map
$$
\xymatrix{
\Hom(QX(c_1)\otimes\cdots\otimes QX(c_n), QX(c)) \ar[d]^{(i_{\mathbf{X}})_*} \\
\Hom(QX(c_1)\otimes\cdots\otimes QX(c_n), X(c))
}
$$
is a trivial fibration. The case $n=0$ is true by assumption (\ref{unit_assumption}). For $n\ge 1$ the result follows from the pushout-product axiom and~(\ref{enrichedSM7}).
\end{proof}

\begin{rem}
Note that condition~$(\ref{unit_assumption})$ can be avoided if the unit $I$ is cofibrant in $\M$ or if the $C$-coloured operad $\mathcal{P}$ is \emph{positive}, i.e., $\mathcal{P}(\,; c)=0$ for every $c\in C$.
\end{rem}

\begin{thm}
Let $\M$ and $\N$ be monoidal $\V$-model categories and $F\colon\M\rightleftarrows\N\colon U$ a monoidal Quillen $\V$-adjunction. Let
$$
\xymatrix{
\mathbb{L}F\colon {\rm Ho}(\M^C) \ar@<3pt>[r] & \ar@<2pt>[l] {\rm Ho}(\N^C)\colon \mathbb{R}U
}
$$
denote the derived adjunction.
Let $\mathcal{P}$ be a cofibrant $C$-coloured operad in $\V$ and consider the extension of $F$ and $U$ over $\M^C$ and $\N^C$, respectively.
If $\mathbf{X}$ is a $\mathcal{P}$\nobreakdash-algebra in $\M$ such that $X(c)$ is cofibrant for every $c\in C$, then $ \mathbb{R}U\mathbb{L}F\mathbf{X}$ admits a
homotopy unique $\mathcal{P}$-algebra structure such that the unit of the derived adjunction
$$
\mathbf{X}\longrightarrow \mathbb{R}U\mathbb{L}F\mathbf{X}
$$
is a map of $\mathcal{P}$-algebras.
\label{der_fib}
\end{thm}
\begin{proof}
If $\mathbf{X}$ is a $\mathcal{P}$-algebra, so is $F\mathbf{X}$ since $F$ is a strong symmetric monoidal functor between monoidal $\V$-categories, and the unit of the adjunction
$$
\mathbf{X}\longrightarrow UF\mathbf{X}
$$
is a map of $\mathcal{P}$-algebras ($U$ is a symmetric monoidal functor between $\V$-categories). Now, $F\mathbf{X}$ is cofibrant because $\mathbf{X}$ is cofibrant in $\M^C$ and $F$ is a left Quillen functor.
By Proposition \ref{fib_rep_alg} the fibrant replacement
\begin{equation}
j_{F\mathbf{X}}\colon F\mathbf{X}\longrightarrow R(F\mathbf{X})
\label{eq01}
\end{equation}
has a homotopy unique $\mathcal{P}$-algebra structure such that $j_{F\mathbf{X}}$ is a map of $\mathcal{P}$-algebras. The functor $U$ sends $\mathcal{P}$-algebras to $\mathcal{P}$-algebras and maps of $\mathcal{P}$-algebras to maps of $\mathcal{P}$-algebras, hence applying it to~(\ref{eq01}) we have that
$$
UF\mathbf{X}\longrightarrow UR(F\mathbf{X})=\mathbb{R}U(F\mathbf{X})=\mathbb{R}U\mathbb{L}F\mathbf{X}
$$
is a map of $\mathcal{P}$-algebras (the last equality holds since $\mathbf{X}$ was already cofibrant).
\end{proof}

\begin{thm}
Let $\M$ and $\N$ be monoidal $\V$-model categories and $F\colon\M\rightleftarrows\N\colon U$ a monoidal Quillen $\V$-adjunction. Let
$$
\xymatrix{
\mathbb{L}F\colon {\rm Ho}(\M^C) \ar@<3pt>[r] & \ar@<2pt>[l] {\rm Ho}(\N^C)\colon \mathbb{R}U
}
$$
denote the derived adjunction. Let $\mathcal{P}$ be a cofibrant
$C$-coloured operad in $\V$ and consider the extensions of $F$ and
$U$ over $\M^C$ and $\N^C$, respectively. If $\mathbf{X}$ is a
$\mathcal{P}$\nobreakdash-algebra in $\N$ such that $X(c)$ is
fibrant for every $c\in C$, and  the induced map
\begin{equation}
\Hom(I,  Q(UX(c)))\longrightarrow \Hom(I, UX(c))
\label{unit_cond}
\end{equation}
is a trivial fibration in $\V$ for every $c\in C$ such that $\mathcal{P}(\,;c)\ne 0$, where $Q$ denotes the cofibrant replacement functor in $\M$, then $\mathbb{L}F \mathbb{R}U\mathbf{X}$ admits a
homotopy unique $\mathcal{P}$-algebra structure such that the counit of the derived adjunction
$$
\mathbb{L}F\mathbb{R}U\mathbf{X}\longrightarrow\mathbf{X}
$$
is a map of $\mathcal{P}$-algebras.
\label{der_cof}
\end{thm}

\begin{proof}
The proof is the same as in Theorem~\ref{der_fib} but using Proposition~\ref{cof_rep_alg}.
\end{proof}

\section{Enriched localization and colocalization of algebras}

\subsection{Enriched Bousfield localizations and colocalizations}
Let $\V$ be a mo\-noi\-dal model category and let $\M$ be a
$\V$-model category. We denote by $\Map(-,-)$ an \emph{enriched
homotopy function complex} defined as
$$
\Map(X,Y)=\Hom(QX, RX)
$$
for every $X$ and $Y$, where $\Hom(-,-)$ denotes the $\V$-enrichment and $Q$ and $R$ are the cofibrant replacement functor and the fibrant replacement functor in $\M$, respectively. Thus, $\Map(X,Y)$ is always a fibrant object in $\V$ that is homotopy invariant.

Enriched Bousfield localizations and colocalizations are similar to Bousfield localizations and colocalizations, but
defining the class of local and colocal objects by means of the
enriched homotopy function complex instead of the simplicial one; see \cite{Bar10} for a detailed account on enriched Bousfield localizations.

Let $\mathcal{L}$ be a class of morphism in $\M$ and let $\mathcal{K}$ be a class of objects in $\M$.
An object $Z$ of $\M$ is \emph{$\mathcal{L}$-local} if it is fibrant and the induced map
$$
f^*\colon\Map(B,Z)\longrightarrow \Map(A, Z)
$$
is a weak equivalence in $\V$ for every map $f\colon A\rightarrow B$ in $\mathcal{L}$.
A map $g\colon X\rightarrow Y$ is an \emph{$\mathcal{L}$-local equivalence} if the induced map
$$
g^*\colon\Map(Y,Z)\longrightarrow \Map(X, Z)
$$
is a weak equivalence in $\V$ for every $Z$ that is $\mathcal{L}$-local.

A map $g\colon X\rightarrow Y$ is a \emph{$\mathcal{K}$-colocal equivalence} if the induced map
$$
g_*\colon\Map(Z,X)\longrightarrow \Map(Z, Y)
$$
is a weak equivalence in $\V$ for every $Z$ in $\mathcal{K}$.
An object $Z$ of $\M$ is \emph{$\mathcal{K}$-colocal} if it is cofibrant and the induced map
$$
f_*\colon\Map(Z,A)\longrightarrow \Map(Z, B)
$$
is a weak equivalence in $\V$ for every map $f\colon A\rightarrow B$ that is a $\mathcal{K}$-colocal equivalence.

\begin{defn}
The \emph{enriched Bousfield localization} of $\M$ with respect to $\mathcal{L}$ is a $\V$-model structure $\M_{\mathcal{L}}$ on $\M$ such that:
\begin{itemize}
\item[(i)] The cofibrations in $\M_{\mathcal{L}}$ are the same as the cofibrations in $\M$.
\item[(ii)] The weak equivalences in $\M_{\mathcal{L}}$ are the $\mathcal{L}$-local equivalences.
\item[(iii)] The fibrant objects in $\M_{\mathcal{L}}$ are the $\mathcal{L}$-local objects.
\end{itemize}
\end{defn}

The \emph{$\mathcal{L}$-localization functor} in $\M$ is defined
as the functorial fibrant replacement $(L,\eta)$ in
$\M_{\mathcal{L}}$. For every $X$ in $\M$, the natural map
$\eta_X\colon X\rightarrow LX$ is a cofibration and an
$\mathcal{L}$-local equivalence, and $LX$ is $\mathcal{L}$-local.

\begin{defn}
The \emph{enriched Bousfield colocalization} of $\M$ with respect to $\mathcal{K}$ is a $\V$-model structure $\M^{\mathcal{K}}$ on $\M$ such that:
\begin{itemize}
\item[(i)] The fibrations in $\M^{\mathcal{K}}$ are the same as the fibrations in $\M$.
\item[(ii)] The weak equivalences in $\M^{\mathcal{K}}$ are the $\mathcal{K}$-colocal equivalences.
\item[(iii)] The cofibrant objects in $\M^{\mathcal{L}}$ are the $\mathcal{K}$-colocal objects.
\end{itemize}
\end{defn}

The \emph{$\mathcal{K}$-colocalization functor} in $\M$ is defined as the functorial cofibrant replacement $(K,\varepsilon)$ in $\M^{\mathcal{K}}$. For every $X$ in $\M$, the natural map $\varepsilon_X\colon KX\rightarrow X$ is a fibration and a $\mathcal{K}$-colocal equivalence, and $KX$ is $\mathcal{K}$-colocal.

As proved in \cite{Bar10}, the enriched Bousfield
localization with respect to a set of morphisms always exists in a
$\V$-model category $\M$, provided that the category $\M$ is left proper and combinatorial, and the category $\V$ is
combinatorial. Dually, the enriched Bousfield colocalization with respect to a set of objects always exists under the same assumptions on $\V$ and $\M$ as above, but replacing left proper by right proper. This is the case, for example, of the category of simplicial sets and the category of symmetric spectra enriched over itself or enriched over simplicial sets.

\begin{rem}
If $\M$ is a monoidal $\V$-model category and $\mathcal{L}$ is a class of morphism and $\mathcal{K}$ a class of objects for which Bousfield localization and colocalization exist, then both $\M_{\mathcal{L}}$ and $\M^{\mathcal{K}}$ are $\V$-model categories, and although they are monoidal as categories, they are \emph{not} always monoidal model categories. Hence $\M_{\mathcal{L}}$ and $\M^{\mathcal{K}}$ are \emph{not} monoidal $\V$\nobreakdash-categories since the pushout-product axiom will not hold in general. However, if $\V=\M$ and we view $\M$ as a monoidal $\M$-model category then $\M_{\mathcal{L}}$ and $\M^{\mathcal{K}}$ are, in fact, monoidal $\M$-model categories.
\label{not_mon_mod}
\end{rem}

The following result follows directly from the definitions of the enriched localized and enriched colocalized model structure.
\begin{lem}
Let $\M$ be a monoidal $\V$-model category. Then the identity functors ${\rm Id}\colon\M\rightleftarrows \M_{\mathcal{L}}\colon {\rm Id}$ and ${\rm Id}\colon\M^{\mathcal{K}}\rightleftarrows \M\colon {\rm Id}
$ are Quillen adjunctions.
\label{Qpairs_lemma}
\end{lem}
\begin{proof}
The cofibrations in $\M_{\mathcal{L}}$ are the same as the
cofibrations in $\M$ and the fibrations in $\M^{\mathcal{K}}$ are
the same as the fibrations in $\M$. The result follows using that
any weak equivalence in $\M$ is both an $\mathcal{L}$-local
equivalence and a $\mathcal{K}$-colocal equivalence.
\end{proof}

\subsection{Localization and colocalization of algebras over coloured operads}
We have seen that enriched localizations and enriched colocalization functors are particular cases of functorial fibrant and functorial cofibrant replacements, respectively.  To prove that these functors preserve algebraic structures in monoidal $\V$\nobreakdash-categories, we would like to apply Theorem~\ref{der_fib} and Theorem~\ref{der_cof} to the Quillen adjunction described in Lemma~\ref{Qpairs_lemma}, but this only work when $\V=\M$ and $\M$ is a monoidal $\M$-category, in which case $\M_{\mathcal{L}}$ and $\M^{\mathcal{K}}$ are monoidal model categories.

For any monoidal $\V$-category $\M$ we know, by Remark~\ref{not_mon_mod},  that $\M_{\mathcal{L}}$ and $\M^{\mathcal{K}}$ are not monoidal model categories in general. However, the proofs that work in the case $\V=\M$ can be adapted to the general case by imposing suitable closure properties on $\mathcal{L}$-local equivalences and $\mathcal{K}$-colocal objects under the tensor product.

\begin{prop}
Let $\mathcal{L}$ be a class of morphisms and $\mathcal{K}$ a class of objects in a monoidal $\V$-model category $\M$ and assume that the enriched Bousfield localization $\M_{\mathcal{L}}$ and the enriched Bousfield colocalization $\M^{\mathcal{K}}$ exist.
\begin{itemize}
\item[{\rm (i)}] Let $L$ be the extension of the localization functor over $\M^C$ and let $\mathbf{X}$ be an object of $\M^C$ such that $X(c)$ is cofibrant for every $c\in C$. If for every $(c_1,\ldots, c_n)$ and $n\ge 0$ the tensor product $\eta_{X(c_1)}\otimes\cdots \otimes\eta_{X(c_n)}$ is an $\mathcal{L}$-local equivalence, then the induced map
$$
\End(L\mathbf{X})\longrightarrow \Hom(\mathbf{X}, L\mathbf{X})
$$
is a trivial fibration of $C$-coloured collections in $\V$.
\item[{\rm (ii)}] Let $K$ be the extension of the colocalization functor over $\M^C$ and let $\mathbf{X}$ be an object of $\M^C$ such that $X(c)$ is fibrant for every $c\in C$. If for every $(c_1,\ldots, c_n)$ and $n\ge 1$ the tensor product $KX(c_1)\otimes\cdots\otimes KX(c_n)$ is $\mathcal{K}$\nobreakdash-colocal and the unit of $\M$ is also $\mathcal{K}$\nobreakdash-colocal, then the induced map
$$
\End(K\mathbf{X})\longrightarrow \Hom(K\mathbf{X},\mathbf{X})
$$
is a trivial fibration of $C$-coloured collections in $\V$.
\end{itemize}
\label{closureprop}
\end{prop}
\begin{proof}
To prove part (i) use that the map
$$
X(c_1)\otimes\cdots\otimes X(c_n)\longrightarrow LX(c_1)\otimes\cdots\otimes LX(c_n)
$$
is an $\mathcal{L}$-local equivalence for every $(c_1,\ldots, c_n)$ and every $n\ge 0$ by assumption. Since $LX(c)$ is $\mathcal{L}$-local for every $c\in C$  and $X(c)$ is cofibrant for every $c\in C$, it follows from the definition of $\mathcal{L}$-local equivalences that
$$
\Hom(LX(c_1)\otimes\cdots\otimes LX(c_n), LX(c))\longrightarrow \Hom(X(c_1)\otimes\cdots\otimes X(c_n), LX(c))
$$
is a weak equivalence.

Part (ii) is proved by a dual argument. For every $(c_1,\ldots, c_n)$ and every $n\ge 0$ the object
$$
KX(c_1)\otimes\cdots\otimes KX(c_n)
$$
is $\mathcal{K}$-colocal by assumption (for $n=0$ the tensor product is the unit $I$). Now, the map $KX(c)\rightarrow X(c)$ is a $\mathcal{K}$-colocal equivalence for every $c\in C$. Therefore, and using that $X(c)$ and hence $KX(c)$ are fibrant for all $c\in C$ we have that
$$
\Hom(KX(c_1)\otimes\cdots\otimes KX(c_n), KX(c))\longrightarrow \Hom(KX(c_1)\otimes\cdots\otimes KX(c_n), X(c))
$$
is a weak equivalence by definition of $\mathcal{K}$-colocal object.
\end{proof}

The closure assumptions of Proposition~\ref{closureprop} are
satisfied when $\V=\M$ as shown in the following lemma.
\begin{lem}
Let $\M$ be a monoidal model category viewed as a monoidal $\M$\nobreakdash-model category and consider the enriched localized and colocalized model structures $\M_{\mathcal{L}}$ and $\M^{\mathcal{K}}$. Then the following hold:
\begin{itemize}
\item[(i)] The tensor product of two $\mathcal{L}$-local equivalences with cofibrant domains and codomains is again an $\mathcal{L}$-local equivalence.
\item[(ii)] The tensor product of two $\mathcal{K}$-colocal object is $\mathcal{K}$-colocal.
\end{itemize}
\label{V=M}
\end{lem}
\begin{proof}
Let $X\rightarrow Y$ and $X'\rightarrow Y'$ be two $\mathcal{L}$-local equivalences, with $X$, $X'$, $Y$ and $Y'$ cofibrant in $\V$. For any $\mathcal{L}$-local object $Z$ we have that
$$
\Hom(Y', Z)\longrightarrow \Hom(X',Z)
$$
is a weak equivalence because $Y'\rightarrow Y$ is an $\mathcal{L}$-local equivalence, $Z$ is $\mathcal{L}$-local, and $Y$ and $Y'$ are cofibrant and $Z$ is fibrant. Therefore
$$
\Hom(Y, \Hom(Y', Z))\longrightarrow \Hom(X, \Hom(X',Z)),
$$
and, by adjunction,
$$
\Hom(Y\otimes Y', Z)\longrightarrow \Hom(X\otimes X',Z),
$$
are weak equivalences. But since $X$, $X'$, $Y$ and $Y'$ are cofibrant and $Z$ is fibrant, the last statement is equivalent to the map
$$
\Map(Y\otimes Y', Z)\longrightarrow \Map(X\otimes X',Z),
$$
being a weak equivalence. A similar argument proves part (ii).
\end{proof}

\begin{thm}
Let $\M$ be a monoidal $\V$-model category. Let $\mathcal{P}$ be a
cofibrant $C$\nobreakdash-coloured operad in $\V$ and consider the
extension of a localization functor $(L,\eta)$  over $\M^C$. If
$\mathbf{X}$ is a $\mathcal{P}$-algebra in $\M$ such that $X(c)$
is cofibrant for every $c\in C$ and the morphism
\begin{equation}
(\eta_{\mathbf{X}})_{c_1}\otimes\cdots\otimes (\eta_{\mathbf{X}})_{c_n}\colon
X(c_1)\otimes\cdots\otimes X(c_n)\longrightarrow LX(c_1)\otimes\cdots\otimes LX(c_n)
\label{condition01}
\end{equation}
is an $\mathcal{L}$-local equivalence for every $n\ge 0$ whenever $\mathcal{P}(c_1,\ldots, c_n; c)$ is nonempty,
then $L\mathbf{X}$ admits a homotopy unique $\mathcal{P}$-algebra structure such that $\eta_{\mathbf{X}}$
is a map of $\mathcal{P}$-algebras.
\label{main_loc}
\end{thm}

\begin{proof}
The localization map of a cofibrant object $\mathbf{X}$ in $\M^C$
is given by the unit of the derived adjunction of the Quillen pair
$$
\xymatrix{
{\rm Id}\colon\M \ar@<3pt>[r] & \ar@<2pt>[l] \M_{\mathcal{L}}\colon {\rm Id}
}
$$
given in Lemma~\ref{Qpairs_lemma}. The result follows along the
same lines as Theorem~\ref{der_fib} but using Lemma~\ref{keylem}
and Proposition~\ref{closureprop}(i) to prove that fibrant
replacements in $\M_{\mathcal{L}}$ preserve $\mathcal{P}$-algebra
structures.
\end{proof}

\begin{thm}
Let $\M$ be a monoidal $\V$-model category. Let $\mathcal{P}$ be a
cofibrant $C$\nobreakdash-coloured operad in $\V$ and consider the
extension of a colocalization functor $(K,\varepsilon)$  over
$\M^C$. If $\mathbf{X}$ is a $\mathcal{P}$-algebra in $\M$ such
that $X(c)$ is fibrant for every $c\in C$, the unit $I$ is
$\mathcal{K}$-colocal, and for every $n\ge 1$
\begin{equation}
KX(c_1)\otimes\cdots\otimes KX(c_n)
\label{condition02}
\end{equation}
are $\mathcal{K}$-colocal whenever $\mathcal{P}(c_1,\ldots, c_n;
c)\ne 0$, then $K\mathbf{X}$ admits a homotopy unique
$\mathcal{P}$-algebra structure such that
$\varepsilon_{\mathbf{X}}$ is a map of $\mathcal{P}$-algebras.
\label{main_coloc}
\end{thm}
\begin{proof}
The colocalization map of a fibrant object $\mathbf{X}$ in $\M^C$
is given by the counit of the derived adjunction of the Quillen
pair
$$
\xymatrix{
{\rm Id}\colon\M^{\mathcal{K}}\ar@<3pt>[r] & \ar@<2pt>[l]  \M\colon {\rm Id}
}
$$
The result follows along the same lines as Theorem~\ref{der_cof} but using Lemma~\ref{keylem} and Proposition~\ref{closureprop}(ii) to prove that cofibrant replacements in $\M^{\mathcal{K}}$ preserve $\mathcal{P}$-algebra structures.
\end{proof}

\begin{rem}
By Lemma~\ref{V=M}, conditions~(\ref{condition01}) and~(\ref{condition02}) are always satisfied in the case $\V=\M$.
\end{rem}

\section{Localization and colocalization away from ideals and coideals}

If $\mathbf{X}=(X(c))_{c\in C}$ is an algebra over a
$C$\nobreakdash-co\-lou\-red operad $\mathcal{P}$, sometimes we
will need to carry out certain constructions on some subset of the
components $X(c)$, but not on the others. For example, any pair
$(R, M)$ of objects of $\M$, where $R$ is a monoid and $M$ is an
$R$-module is an algebra over the 2-\nobreakdash coloured operad
for modules over monoids. It is interesting to localize or
colocalize not only $R$ and $M$ at the same time but also only $R$
or only $M$ and ask in which cases the resulting pair is still a
monoid and a module over it. Hence, it is convenient to introduce
the following convention about extending endofunctors on the
category $\M$ to endofunctors on the product category $\M^C$
relative to a subset of the set of colours~$C$.

\begin{defn}
Let $F$ be an endofunctor on a category $\M$. The \emph{extension
of $F$ over $\M^C$ away from a subset $J\subseteq C$} is the
endofunctor on $\M^C$, which we keep denoting by $F$, given by
$F\mathbf{X} = (F_cX(c))_{c\in C}$ where $F_c$ is the identity
functor if $c\in J$ and $F_c=F$ if $c\not\in J$. If $F$ is
equipped with an augmentation $\eta$, then
$(\eta_{\mathbf{X}})_c=\eta_{X(c)}$ if $c\not\in J$ and the
identity map otherwise. Similarly, if $F$ is equipped with a
coaugmentation $\varepsilon$, then
$(\varepsilon_{\mathbf{X}})_c=\varepsilon_{X(c)}$ if $c\not\in J$
and the identity map otherwise.
\end{defn}

We also define special subsets of the set of colours
(depending on $\mathcal{P}$) called \emph{ideals} and \emph{coideals}.

\begin{defn}
Let $\mathcal{P}$ be a $C$-coloured operad. A subset $J\subseteq
C$ is called an \emph{ideal relative to $\mathcal{P}$} if
$\mathcal{P}(c_1,\ldots, c_n; c)=0$ whenever $n\ge 1$, $c\in J$,
and $c_i\not\in J$ for some $i\in \{1,\ldots, n\}$. A subset
$I\subseteq C$ is called a \emph{coideal relative to
$\mathcal{P}$} if $I=C\setminus J$ for some ideal $J$ relative to
$\mathcal{P}$.
\end{defn}

Ideals and coideal will codify to which of the components of the
$\mathcal{P}$-algebra we can apply the corresponding localization
or colocalization functor and still get a
$\mathcal{P}$\nobreakdash-algebra structure. All of the theorems
in the preceding sections can be generalized by using ideals and
coideals in the set of colours, with similar proofs. As an
example, we formulate the extension of
Theorem~\ref{main_loc} and Theorem~\ref{main_coloc} in this
setting.

\begin{thm}
Let $\M$ be a monoidal $\V$-model category. Let $\mathcal{P}$ be a
cofibrant $C$\nobreakdash-coloured operad in $\V$ and consider the
extension of a localization functor $(L,\eta)$  over $\M^C$ away
from an ideal $J\subseteq C$. If $\mathbf{X}$ is a
$\mathcal{P}$-algebra in $\M$ such that $X(c)$ is cofibrant for
every $c\in C$ and the morphism
$$
(\eta_{\mathbf{X}})_{c_1}\otimes\cdots\otimes (\eta_{\mathbf{X}})_{c_n}\colon
X(c_1)\otimes\cdots\otimes X(c_n)\longrightarrow L_{c_1}X(c_1)\otimes\cdots\otimes L_{c_n}X(c_n)
$$
is an $\mathcal{L}$-local equivalence for every $n\ge 0$ whenever $\mathcal{P}(c_1,\ldots, c_n; c)$ is nonempty,
then $L\mathbf{X}$ admits a homotopy unique $\mathcal{P}$-algebra structure such that $\eta_{\mathbf{X}}$
is a map of $\mathcal{P}$-algebras. $\hfill\qed$
\label{mainthm_ideal}
\end{thm}
\begin{proof}
One proceeds as in the proof of Theorem~\ref{main_loc}. The key point is to check that for all $n\ge 0$ and all $(c_1,\ldots, c_n;c)$ such that $\mathcal{P}(c_1,\ldots, c_n; c)\ne 0$ the induced map
$$
\xymatrix{
\Hom(L_{c_1}X(c_1)\otimes\cdots\otimes L_{c_n}X(c_n), L_cX(c)) \ar[d] \\
\Hom(X(c_1)\otimes\cdots\otimes X(c_n), L_c X(c))
}
$$
is a trivial fibration in $\mathcal{V}$. Indeed, if $c\in J$, then $c_i\in J$ for all $i$ (since $J$ is an ideal) and therefore the map is the identity; and if $c\not\in J$, then it is a weak equivalence since $L_cX(c) = LX(c)$ is $\mathcal{L}$-local and $(\eta_{\mathbf{X}})_{c_1}\otimes\cdots\otimes (\eta_{\mathbf{X}})_{c_n}$ is an $\mathcal{L}$-local equivalence, and it is a fibration by~(\ref{enrichedSM7}). The $\mathcal{P}$-algebra structure on $L\mathbf{X}$ is now obtained similarly as in Lemma~\ref{keylem}(i).
\end{proof}

\begin{thm}
Let $\M$ be a monoidal $\V$-model category. Let $\mathcal{P}$ be a
cofibrant $C$\nobreakdash-coloured operad in $\V$ and consider the
extension of a colocalization functor $(K,\varepsilon)$  over
$\M^C$ away from a coideal $C\setminus J$. If $\mathbf{X}$ is a
$\mathcal{P}$-algebra in $\M$ such that $X(c)$ is fibrant for
every $c\in C$, the unit $I$ is $\mathcal{K}$-colocal if
$\mathcal{P}(\;,c)\ne 0$ for some $c\in C$, and for every $n\ge 1$
$$
K_{c_1}X(c_1)\otimes\cdots\otimes K_{c_n}X(c_n)
$$
are $\mathcal{K}$-colocal whenever $\mathcal{P}(c_1,\ldots, c_n; c)$ is nonempty and $c\not\in C\setminus J$,
then $K\mathbf{X}$ admits a homotopy unique $\mathcal{P}$-algebra structure such that $\varepsilon_{\mathbf{X}}$
is a map of $\mathcal{P}$-algebras. $\hfill\qed$
\label{mainthm_coideal}
\end{thm}
\begin{proof}
For all $n\ge 0$ and all $(c_1,\ldots, c_n;c)$ such that $\mathcal{P}(c_1,\ldots, c_n; c)\ne 0$ the induced map
$$
\xymatrix{
\Hom(K_{c_1}X(c_1)\otimes\cdots\otimes K_{c_n}X(c_n), K_cX(c)) \ar[d] \\
\Hom(K_{c_1}X(c_1)\otimes\cdots\otimes K_{c_n}X(c_n), X(c))
}
$$
is a trivial fibration in $\mathcal{V}$. Indeed, if $c\in C\setminus J$, then the map is the identity. If $c\not\in C\setminus J$, then $c_i\not\in C\setminus J$ for all $i$ (since $C\setminus  J$ is a coideal) and therefore it is a weak equivalence since $K_c X(c)=KX(c)$ is $\mathcal{K}$-colocal and $K_{c_1}X(c_1)\otimes\cdots\otimes K_{c_n}X(c_n)$ is $\mathcal{K}$-colocal by assumption; and it is a fibration by~(\ref{enrichedSM7}). The $\mathcal{P}$-algebra structure on $K\mathbf{X}$ is now obtained similarly as in Lemma~\ref{keylem}(ii).
\end{proof}

\begin{rem}
Note that Theorem~\ref{mainthm_coideal} remains true if we replace
the coideal $C\setminus J$ by an ideal $J$ and we impose that
$X(c)$ is also cofibrant for every $c\in J$. This additional
assumption ensures that $K_{c_1}X(c_1)\otimes\cdots\otimes
K_{c_n}X(c_n)$ is always cofibrant if $\mathcal{P}(c_1,\ldots,
c_n; c)\ne 0$ and $c\not\in J$. \label{colocal_ideals}
\end{rem}

\section{Connective covers of $A_{\infty}$ and $E_{\infty}$ rings and modules}

Let $\Sp$ denote the category of symmetric spectra with the
standard model structure given in \cite[\S3.4]{HSS00}. This is a
monoidal $\sSets_*$-model category, where $\sSets_*$ denotes the
category of pointed simplicial sets.

For any $n\in\Z$, we say that a spectrum $X$ is \emph{$n$-connective} if $\pi_k(X)=0$ for $k\le n$. We call $X$ \emph{connective} if it is $(-1)$-connective.
Let $\mathcal{K}(n)=\{\Sigma^{n+1} S^0\}$, where $S^0$ denotes the sphere spectrum. The $n$-connective cover functor $K_{n+1}$ is the enriched $\mathcal{K}(n)$-colocalization functor in $\Sp$ (where we view $\Sp$ as a category enriched over simplicial sets). For any spectrum $X$, there is a natural map
$$
\varepsilon_{X}\colon K_{n+1}X\longrightarrow X
$$
such that $K_{n+1}X$ is $n$-connective and cofibrant and $\varepsilon_{X}$ is a $\mathcal{K}(n)$-colocal equivalence, that is, it induces an isomorphism in $\pi_k$ for $k>n$. Any cofibrant $n$-connective spectrum is of the form $K_{n+1}X$ for some $X$.

Let $\map(-,-)$ denote the simplicial homotopy function complex in $\Sp$ defined by $\map(X,Y)=\Map(QX, RY)$, where $Q$ is the cofibrant replacement functor, $R$ is the fibrant replacement functor and $\Map(-,-)$ is the simplicial enrichment of~$\Sp$. Then
\[
\pi_n\map(X,Y)\cong [\Sigma^nX, Y]\cong \pi_n F(X,Y)
\]
for all spectra $X$, $Y$ and $n\ge 0$, where $F(-,-)$ denotes the
derived function spectrum; cf.\ \cite[Lemma~6.1.2]{Hov99}. Thus,
the simplicial set $\map(X,Y)$ has the same homotopy groups (with
any choice of a base point) as the connective cover $F^c(X,Y)$ of
the spectrum $F(X,Y)$. Hence, if we have a colocalization on~$\Sp$
with respect to a set of objects $\mathcal{K}$, then a (cofibrant)
spectrum $X$ is $\mathcal{K}$-colocal if and only if
\begin{equation}
\label{Fc}
F^c(X, f)\colon F^c(X,A)\longrightarrow F^c(X,B)
\end{equation}
is a weak equivalence of spectra for every $\mathcal{K}$\nobreakdash-colocal equivalence $A\rightarrow B$.

\begin{rem}
Note that, by (\ref{Fc}), if $X$ is $\mathcal{K}$-colocal, then $X$ is also $\Sigma^{-k}\mathcal{K}$-colocal, where $\Sigma^{-k} \mathcal{K}=\{\Sigma^{-k} Z\mid Z\in\mathcal{K}\}$, and (a cofibrant replacement of)~$\Sigma^k X$ is $\mathcal{K}$\nobreakdash-colocal, for all $k\ge 0$. Dually, if $f$ is a $\mathcal{K}$-colocal equivalence, then $f$ is also a $\Sigma^k\mathcal{K}$-colocal equivalence and $\Sigma^{-k} f$ is a $\mathcal{K}$-colocal equivalence, for all $k\ge 0$. Moreover, since $F^c(X,A)\cong F^c(\Sigma^k X, \Sigma^k A)$ for all $X$, $A$ and $k\in \mathbb{Z}$, we may infer that if $X$ is $\mathcal{K}$-colocal, then $\Sigma^k X$ is $\Sigma^k \mathcal{K}$-colocal for all $k\in\Z$, and similarly for $\mathcal{K}$-colocal equivalences.
\label{susp_colocal}
\end{rem}

The smash product of two $\mathcal{K}$\nobreakdash-colocal spectra need not be $\mathcal{K}$\nobreakdash-colocal, but it is so if any of them is connective. The following lemma is a generalization of this fact when $\mathcal{K}=\mathcal{K}(n)$ for some $n\in\Z$.

\begin{lem}
For every $n$, $m\in \Z$, if  $X$ is $\mathcal{K}(n)$-colocal and $Y$ is $\mathcal{K}(m)$-colocal, then the spectrum $X\wedge Y$ is $\mathcal{K}(n+m)$-colocal, . In particular, if $X$ is cofibrant and connective, then $X\wedge K_{n} Y$ is $\mathcal{K}(n)$\nobreakdash-colocal.
\label{wedge_colocal}
\end{lem}
\begin{proof}
Let $A\rightarrow B$ be any $\mathcal{K}(n+m)$-colocal equivalence. By Remark~\ref{susp_colocal}, the spectrum $\Sigma^{-n}X$ is $\mathcal{K}(0)$-colocal. Then
\begin{multline}\notag
F^c(X\wedge Y, A)\simeq F^c(\Sigma^{-n}X\wedge \Sigma^n Y, A)\simeq F^c(\Sigma^{-n} X, F^c(\Sigma^n Y, A)) \\ \notag
\simeq F^c(\Sigma^{-n} X, F^c(\Sigma^n Y, B))\simeq F^c(\Sigma^{-n}X\wedge \Sigma^n Y, B)\simeq F^c(X\wedge Y, B).
\end{multline}
The equivalence $F^c(\Sigma^n Y, A)\simeq F^c(\Sigma^n Y, B)$ holds since $\Sigma^{n}Y$ is $\mathcal{K}(n+m)$-colocal by Remark~\ref{susp_colocal} and $A\rightarrow B$ is a $\mathcal{K}(n+m)$-colocal equivalence.
\end{proof}

An $A_{\infty}$ ring spectrum is an algebra over a cofibrant replacement of the associative operad $\mathcal{A}{\rm ss}$ in simplicial sets. An $E_{\infty}$ ring spectrum is an algebra over a cofibrant replacement of the commutative operad $\mathcal{C}{\rm om}$ in simplicial sets.

\begin{thm}
Let $R$ be an $A_{\infty}$ or $E_{\infty}$ ring spectrum with a fibrant underlying spectrum. Then, the connective cover $K_0 R$ has a homotopy unique $A_{\infty}$ or $E_{\infty}$ ring spectrum structure such that the colocalization map $K_0 R\rightarrow R$ is a map of $A_{\infty}$ or $E_{\infty}$ ring spectra.
\end{thm}
\begin{proof}
By Lemma~\ref{wedge_colocal}, we have that $K_0X\otimes\stackrel{(n)}{\cdots} \otimes K_0X$ is $\mathcal{K}(0)$-colocal for every $n\ge 1$. For $n=0$ we know that the sphere spectrum is cofibrant in the standard model structure of symmetric spectra. The result now follows from Theorem~\ref{main_coloc} since $A_{\infty}$ and $E_{\infty}$ are cofibrant operads.
\end{proof}

Let $\mathcal{P}$ be a one coloured operad. Then, there is a $C$-coloured operad $\mathcal{M}{\rm od}_{\mathcal{P}}$ with two colours $C=\{r,m\}$ whose algebras are pairs $(R,M)$ where $R$ is a $\mathcal{P}$\nobreakdash-algebra and $M$ is a (left) $\mathcal{P}$-module. An $A_{\infty}$ or $E_{\infty}$ module (over an $A_{\infty}$ or $E_{\infty}$ ring spectrum) is an algebra over a cofibrant resolution of the 2-coloured operad $\mathcal{M}{\rm od}_{\mathcal{A}{\rm ss}}$ or $\mathcal{M}{\rm od}_{\mathcal{C}{\rm om}}$, respectively; see~\cite[\S2.1 and \S5.1]{CGMV10} for a description of these coloured operads.

\begin{thm}
Let $(R,M)$ be an $A_{\infty}$ or $E_{\infty}$ module where $R$ and $M$ are fibrant as spectra. Then $(K_0 R, K_n M)$ has a homotopy unique $A_{\infty}$ or $E_{\infty}$ module structure such that the colocalization map is a map of $A_{\infty}$ or $E_{\infty}$ modules, for every $n\in \Z$.
\end{thm}
\begin{proof}
The $C$-coloured operad for modules over monoids has two colours $C=\{r,m\}$. The set $\{r\}$ is an ideal and the set $\{m\}$ is a coideal for this coloured operad. Applying Theorem~\ref{mainthm_coideal} using the colocalization functor $K_0$ and the coideal $\{m\}$ we obtain that $(K_0 R, M)$ is an $A_{\infty}$ or $E_{\infty}$ module. We may now use Remark~\ref{colocal_ideals} (note that $K_0 R$ is both fibrant and cofibrant as a spectrum) and apply Theorem~\ref{mainthm_coideal} with the colocalization functor $K_n$ and the ideal $\{r\}$ to the pair $(K_0 R, M)$, yielding the desired result.
\end{proof}

\end{document}